\documentclass[fleqn,reqno,11pt,a4paper,final]{amsart}

\usepackage[a4paper,left=30mm,right=30mm,top=30mm,bottom=30mm,marginpar=20mm]{geometry}
\usepackage{amsmath}
\usepackage{amssymb}
\usepackage{amsthm}
\usepackage{amscd}
\usepackage[ansinew]{inputenc}
\usepackage{cite}
\usepackage{bbm}
\usepackage{color}
\usepackage[english=american]{csquotes}
\usepackage[final]{graphicx}
\usepackage{hyperref}
\usepackage{calc}
\usepackage{mathptmx}
\usepackage[T1]{fontenc}

\graphicspath{{../Pictures/}}

\numberwithin{equation}{section}

\newtheoremstyle{thmlemcorr}{10pt}{10pt}{\itshape}{}{\bfseries}{.}{10pt}{{\thmname{#1}\thmnumber{ #2}\thmnote{ (#3)}}}
\newtheoremstyle{thmlemcorr*}{10pt}{10pt}{\itshape}{}{\bfseries}{.}\newline{{\thmname{#1}\thmnumber{ #2}\thmnote{ (#3)}}}
\newtheoremstyle{remexample}{10pt}{10pt}{}{}{\bfseries}{.}{10pt}{{\thmname{#1}\thmnumber{ #2}\thmnote{ (#3)}}}
\newtheoremstyle{ass}{10pt}{10pt}{}{}{\bfseries}{.}{10pt}{{\thmname{#1}\thmnumber{ A#2}\thmnote{ (#3)}}}

\theoremstyle{thmlemcorr}
\newtheorem{theorem}{Theorem}
\numberwithin{theorem}{section}
\newtheorem{lemma}[theorem]{Lemma}
\newtheorem{corollary}[theorem]{Corollary}

\theoremstyle{thmlemcorr*}
\newtheorem{theorem*}{Theorem}
\newtheorem{lemma*}[theorem]{Lemma}
\newtheorem{corollary*}[theorem]{Corollary}
\newtheorem{proposition*}[theorem]{Proposition}
\newtheorem{problem*}[theorem]{Problem}
\newtheorem{conjecture*}[theorem]{Conjecture}
\newtheorem{definition*}[theorem]{Definition}

\theoremstyle{remexample}
\newtheorem{remark}[theorem]{Remark}

\theoremstyle{ass}

\newcommand{\Ocal}{\mathcal{O}}

\newcommand{\T}{\mathbb{T}}

\DeclareMathOperator{\diverg}{div}

\DeclareMathOperator{\supp}{supp}

\newcommand{\norm}[1]{\|#1\|}

\newcommand{\R}{\mathbb{R}}

\newcommand{\term}[1]{\textbf{#1}}

\def\XXint#1#2#3{{\setbox0=\hbox{$#1{#2#3}{\int}$}
\vcenter{\hbox{$#2#3$}}\kern-.5\wd0}}

\renewcommand{\epsilon}{\varepsilon}
\renewcommand{\phi}{\varphi}

\begin{document}


\title[Localised Relative Energy]{Localised Relative Energy and Finite Speed of Propagation for Compressible Flows}

\author{Emil Wiedemann}
\address{\textit{Emil Wiedemann:} Institute of Applied Mathematics, Leibniz University Hannover, Welfengarten~1, 30167 Hannover, Germany}
\email{wiedemann@ifam.uni-hannover.de}

\begin{abstract}
 For the incompressible and the isentropic compressible Euler equations in arbitrary space dimension, we establish the principle of localised relative energy, thus generalising the well-known relative energy method. To this end, we adapt classical arguments of C.~Dafermos to the Euler equations. We give several applications to the behaviour of weak solutions, like local weak-strong uniqueness, local preservation of smoothness, and finite speed of propagation for the isentropic system.
\end{abstract}







\maketitle




\section{Introduction}
This paper is mainly about the \term{isentropic compressible Euler equations} in arbitrary space dimension $d\geq1$, and about their weak solutions that satisfy the local energy inequality (or, in hyperbolic terminology, the entropy inequality). The system is given as
\begin{equation*}
\begin{aligned}
\partial_t (\rho u) + \diverg(\rho u\otimes u)+\nabla \rho^\gamma&=0,\\
\partial_t\rho+\diverg (\rho u)&=0,
\end{aligned}
\end{equation*} 
and is often considered a paradigmatic hyperbolic system of conservation laws in multiple space dimensions. It arises from the full Euler system, which consists of the conservation of mass, momentum, and total energy of an inviscid compressible fluid, by assuming the entropy to be constant. Here, for simplicity, we consider only polytropic gases, for which the pressure law reads $p(\rho)=\rho^\gamma$ with $\gamma>1$. The variable $\rho$ signifies the density of the gas and $u$ its velocity. As long as $\rho>0$ remains bounded away from zero, the system is strictly hyperbolic, whereas the presence of vacuum sets the Euler equations outside of the standard theory of hyperbolic conservation laws. 

The study of the Euler equations in more than one space dimension is notoriously difficult, particularly when it comes to rigorous mathematical results. To date there is no known class of solutions that exist globally in time for any initial data, except for {measure-valued solutions} and related ``very weak" solutions. On the other hand, the question of uniqueness of solutions has been challenging researchers for decades. It is well-known that even for the simplest examples of scalar nonlinear conservation laws, several weak solutions may arise from the same initial data, and that the unique ``correct" solution has to be singled out invoking the Second Law of Thermodynamics. At least for scalar conservation laws, Kruzhkov made this rigorous in his seminal work~\cite{kruzhkov}, proving that weak solutions satisfying an \emph{entropy inequality} exist and are unique. 

For hyperbolic systems, however, such a result can not be true, as shown by recent counterexamples to uniqueness for the isentropic Euler system~\cite{euler2, chiodaroli, chiodarolidelelliskreml, klingenberg} and related models~\cite{chiodarolifeireislkreml, donatelli, feireislgwiazdasavage, luoxiexin, chiodarolimichalek}. Despite the lack of a proper uniqueness criterion, a weaker form of uniqueness does hold true: So-called \term{weak-strong uniqueness} ensures that an entropy solution remains unique \emph{as long as a strong solution with the same initial data exists}. This principle is valid for surprisingly large classes of solutions, including measure-valued ones; see for example \cite{Daf, brenieretal, FeJiNo, DST, GwSwWi, brezinafeireisl, WiSurvey, DeGwLy}, just to name a few.   

The principle of weak-strong uniqueness is usually established by the \term{relative energy/entropy method}, introduced by C.~Dafermos~\cite{Daf}. The idea is to compare a strong and a weak solution in terms of a relative version of the formally conserved energy or entropy of the system in question, and to derive an estimate for this quantity which forces it to vanish identically if it was zero initially. In the course of the estimate, it is typically necessary to shift all appearing derivatives to the strong solution. Relative energy or entropy approaches do not only help to establish weak-strong uniqueness, but can also be applied to study unconditional uniqueness~\cite{FeKrVa}, stability~\cite{chen1, chen2, serre, serrevasseur}, singular limits~\cite{saintraymond,feireisllukacova, brezinamacha}, or asymptotic behaviour at large times~\cite{perthame}.

In this article we introduce a space-localised form of relative energy for the incompressible and isentropic Euler systems, based on the relative entropy method in the version established by Dafermos in Theorem 5.3.1 of~\cite{Daf}. As an application, we prove \term{finite speed of propagation} for energy-admissible solutions of the isentropic Euler equations. The principle of finite speed of propagation says, in its basic version, that a solution whose support\footnote{In the context of compressible flows given by a pair $(\rho,u)$, by ``support" we really mean the support of $(\rho-\bar{\rho},u)$, where $\bar{\rho}>0$ is a constant ``background density".} is initially contained in a ball with radius $\eta$ will stay in a ball of radius at most $\eta+Ct$. For smooth solutions this can be proved without much difficulty, see e.g.\ Theorem 2.2 in~\cite{wang}. For weak solutions with density bounded below by a positive constant, the result follows from the above-mentioned Theorem 5.3.1 in~\cite{Daf}. In this paper we are also able to deal with solutions that possibly contain vacuum regions, at least for adiabatic exponents $\gamma\geq2$.  

While it may seem perfectly plausible that weak solutions will obey the principle of finite-speed propagation, which represents a fundamental characteristic of hyperbolic equations, it is not at all obvious that this can be rigorously proved. After all, the above-mentioned counterexamples for the isentropic system show that weak solutions may behave in unexpected ways. In fact, finite speed of propagation is \emph{not} valid for weak solutions in the absence of an energy inequality, as demonstrated in Chapter 5 of \cite{chiodarolithesis}: An initially compactly supported solution (i.e.\ a solution whose velocity is zero and whose density is constant outside a compact set) may instantaneously occupy the whole space. Indeed, a particular instance of Theorem 5.1.1 in~\cite{chiodarolithesis} is given by zero initial velocity and constant initial density, from which non-trivial solutions emerge. 

If the local energy inequality is satisfied, however, we can show a stronger version of the principle of finite-speed propagation (Corollary \ref{finite} below): \emph{Given a smooth solution and a bounded weak solution that initially agrees with the strong one outside a ball of radius $R$, then the two solutions remain identical outside a ball of radius $R+Ct$}. In other words, possibly non-smooth perturbations to a solution can spread only at finite speed.

This statement is an easy consequence of \term{local weak-strong uniqueness}, which may be interesting in its own right: \emph{If a strong and a weak solution agree initially in a neighbourhood of a point, then they agree on a (possibly smaller) neighbourhood of that point up to a positive time} (Theorem \ref{locweakstrong}). Again, in the absence of vacuum, this follows from the result in~\cite{Daf}, whereas here we can deal with vacuum if $\gamma\geq2$.

We also consider the \emph{incompressible} Euler equations
\begin{equation*}
\begin{aligned}
\partial_t u + \diverg(u\otimes u)+\nabla p&=0,\\
\diverg u&=0.
\end{aligned}
\end{equation*}
These equations are nonlocal, because the pressure depends on the velocity by means of a singular integral operator.  Consequently, perturbations do not propagate at finite speed -- not even for perfectly smooth solutions. Correspondingly, our localised relative energy approach gives only a weaker result than in the compressible case: A strong and a weak solution which share their initial data in some open set cannot instantaneously ``jump away" from each other on that set (Corollary~\ref{incompcor}) in terms of the $L^2$ norm. We remark that Corollary~\ref{incompcor} is not trivial and does not hold true in absence of the energy inequality; indeed, the conclusion of the corollary is violated in case of the Scheffer-Shnirelman paradox~\cite{scheffer, shnirel1} and related examples~\cite{euler1, eulerexistence}. In this context, it is an interesting open question whether ``turbulent zones" in the incompressible Euler equations with vortex sheet initial data can spread only at finite speed, as suggested by the results in~\cite{vortexsheet, eulerboundary}. 

The paper proceeds as follows: In Section~\ref{prelim} we recall the definition of weak solutions for the incompressible and compressible Euler equations and the respective local energy inequality. Section \ref{locrelenergy} presents the localised relative energy method for both Euler systems: Theorems~\ref{basicthm} and~\ref{Eweak-strong} give estimates for the relative energy restricted to a subdomain (namely, the support of a smooth cutoff function $\phi$). The flux of relative energy through this (possibly time-dependent) subdomain results in two additional integrals compared to the conventional (global) relative energy. Section~\ref{appl} then features the mentioned applications to local weak-strong-uniqueness, finite speed of propagation, and local preservation of smoothness, obtained from Theorems~\ref{basicthm} and~\ref{Eweak-strong} by choosing $\phi$ appropriately. This requires boundedness of solutions and either absence of vacuum or $\gamma\geq2$. In view of conceivable applications to singular limits, where boundedness of approximate solutions is not guaranteed by the available a priori estimates, it would be interesting to extend our results to weak solutions that are not necessarily in $L^\infty$.

\textbf{Acknowledgement.} I would like to thank Jan Giesselmann for valuable comments on a previous version of this manuscript.

\section{The Euler Equations}\label{prelim}
Throughout this paper, $\Omega\subset\R^d$ denotes a (not necessarily bounded) domain and $T>0$ any (not necessarily ``small") time. 

\subsection{The incompressible system}
Recall that the incompressible Euler equations are given as
\begin{equation}\label{euler}
\begin{aligned}
\partial_t u + \diverg(u\otimes u)+\nabla p&=0\\
\diverg u&=0.
\end{aligned}
\end{equation}
Often, weak solutions are defined by testing against divergence-free test fields only, to the effect that the pressure can be ignored and one only needs to assume $u\in L_{loc}^2(\Omega\times (0,T))$ in order to make sense of the resulting integral. However, as observed by Duchon-Robert \cite{duchonrobert}, if the local energy flux is to be well-defined, one should rather consider pairs 
\begin{equation*}
(u,p)\in L^3_{loc}(\Omega\times[0,T];\R^d)\times L_{loc}^{3/2}(\Omega\times[0,T];\R),
\end{equation*}
and such a pair is called a \term{weak solution} of this system with initial datum $u^0\in L^2_{loc}(\Omega)$ if, for almost every $\tau\in (0,T)$ and every vector field $\phi\in C_c^1(\Omega\times [0,T];\R^d)$, we have  
\begin{equation*}
\begin{aligned}
&\int_0^\tau\int_{\Omega}\partial_t\phi(x,t)\cdot u(x,t)+\nabla\phi(x,t):(u\otimes u)(x,t)+\diverg\phi(x,t)p(x,t)dxdt\\
&\hspace{3cm}=\int_{\Omega}u(x,\tau)\cdot\phi(x,\tau)-u^0(x)\cdot\phi(x,0)dx,
\end{aligned}
\end{equation*}
and if $u$ is weakly divergence-free in the sense that
\begin{equation*}
\int_{\Omega}u(x,\tau)\cdot\nabla \psi(x)dx=0
\end{equation*}
for almost every $\tau\in(0,T)$ and every $\psi\in C_c^1(\Omega)$. Note we have not imposed any boundary condition, as we are interested only in the \emph{local} properties of a solution within $\Omega$.

A weak solution $u$ satisfies the \term{local energy inequality} if 
\begin{equation*}
\frac12\partial_t|u|^2+\diverg\left[\left(\frac{|u|^2}{2}+p\right)u\right]\leq0
\end{equation*} 
in the sense of distributions, i.e.\ if
\begin{equation*}
\int_0^\tau\int_\Omega \frac12\partial_t\phi|u|^2+\nabla\phi\cdot\left[\left(\frac{|u|^2}{2}+p\right)u\right]dxdt\geq \frac12\int_\Omega \phi(\tau)|u(\tau)|^2dx-\frac12\int_\Omega \phi(0)|u^0|^2dx
\end{equation*}
for every nonnegative $\phi\in C_c^1(\Omega\times[0,T])$, for almost every $\tau\in(0,T)$. This reflects the absence of local energy production and requires $u\in L^3_{loc}$, $p\in L_{loc}^{3/2}$, hence the corresponding assumption in the definition of weak solution. Local admissibility is the natural analogue for the incompressible Euler equations of the so-called \term{entropy condition} in the context of (hyperbolic) conservation laws. 

\subsection{The isentropic compressible system}
The isentropic compressible Euler equations with adiabatic exponent $\gamma>1$ are given as
\begin{equation}\label{isentropic}
\begin{aligned}
\partial_t (\rho u) + \diverg(\rho u\otimes u)+\nabla \rho^\gamma&=0,\\
\partial_t\rho+\diverg (\rho u)&=0.
\end{aligned}
\end{equation}

Again we can formulate these equations in the sense of distributions: If $\rho^0\in L^\gamma_{loc}(\Omega)$ is nonnegative and $\rho^0|u^0|^2\in L^1_{loc}(\Omega)$, then a pair $(\rho,u)$ is a \term{weak solution} of the isentropic system with initial datum $(\rho^0,u^0)$ if $\rho\in L_{loc}^\gamma(\Omega\times[0,T])$, $\rho|u|^2\in L^1_{loc}(\Omega\times [0,T])$, and

\begin{equation}\label{Emass_momentum}
\begin{aligned}
\int_0^\tau\int_{\Omega}\partial_t\psi \rho+\nabla\psi\cdot{\rho u}dxdt+\int_{\Omega}\psi(x,0)\rho^0-\psi(x,\tau)\rho(x,\tau)dx&=0,\\
\int_0^\tau\int_{\Omega}\partial_t\eta\cdot{\rho u}+\nabla\eta : {(\rho u\otimes u)}+\diverg\eta{\rho^\gamma}dxdt\\
+\int_{\Omega}\eta(x,0)\cdot \rho^0u^0-\eta(x,\tau)\cdot{\rho u}(x,\tau)dx&=0
\end{aligned}
\end{equation}
for almost every $\tau\in(0,T)$, every $\psi\in C_c^1(\Omega\times[0,T])$, and every $\eta\in C_c^1(\Omega\times[0,T];\R^d)$. In analogy with the incompressible situation, we say the solution fulfills the \term{local energy inequality} if for almost every $\tau\in(0,T)$

\begin{equation*}
\partial_t\left(\frac{\rho|u|^2}{2}+\frac{1}{\gamma-1}\rho^\gamma\right)+\diverg\left[\left(\frac{\rho|u|^2}{2}+\frac{\gamma}{\gamma-1}\rho^\gamma\right)u\right]\leq0
\end{equation*}
in the sense of distributions, i.e.
\begin{equation}\label{Emvsenergy}
\begin{aligned}
\int_\Omega&\frac12\phi(\tau)\rho(\tau)|u|^2(\tau)+\frac{1}{\gamma-1}\phi(\tau)\rho(\tau)^\gamma dx\\
\leq&\int_\Omega\frac{1}{2}\phi(0)\rho^0|u^0|^2+\frac{1}{\gamma-1}\phi(0)(\rho^0)^\gamma dx\\
&+\int_0^\tau\int_\Omega\partial_t\phi\left(\frac{\rho|u|^2}{2}+\frac{1}{\gamma-1}\rho^\gamma\right)+\left(\frac{\rho|u|^2}{2}+\frac{\gamma}{\gamma-1}\rho^\gamma\right)\nabla\phi\cdot u dxdt
\end{aligned}
\end{equation}
for every nonnegative $\phi\in C_c^1(\Omega\times[0,T])$. In particular, by stating that a solution satisfies the local energy inequality, we imply that all integrals appearing in~\eqref{Emvsenergy} are well-defined. When $\rho$ is bounded away from zero, then \eqref{Emass_momentum} becomes a hyperbolic system of conservation laws, and \eqref{Emvsenergy} constitutes the entropy inequality. One should not be confused by the fact that, in this case, the (physical) energy coincides with the (mathematical) entropy.

\section{Localised Relative Energy}\label{locrelenergy}
\subsection{The incompressible case}
In this section we demonstrate the localised relative energy method in the simpler case of the incompressible Euler equations \eqref{euler}.

We will make use of the following terminology: Given a relatively open subset $\mathcal{O}\subset\Omega\times[0,T]$, we say that a pair $(U,P)\in C^1(\overline{\mathcal{O}})$ is a \term{local strong solution} in $\mathcal{O}$ if it satisfies the Euler equations pointwise at every $(x,t)\in \mathcal{O}$, and a pair $(u,p)\in L^3_{loc}(\mathcal{O})\times L^{3/2}_{loc}(\mathcal{O})$ is a \term{local weak solution} in $\mathcal{O}$ if it satisfies the definition of weak solution for every test function with compact support in~$\mathcal{O}$. Note these functions need not even be defined outside $\Ocal$ to qualify as local (strong or weak) solutions in $\Ocal$. 

The local energy inequality can be interpreted analogously on $\mathcal{O}$, so that the inequality only needs to hold when tested against functions of compact support in $\Ocal$. A local weak solution in $\Ocal$ that satisfies the local energy inequality in $\Ocal$ is called an \term{admissible weak solution in $\Ocal$}. Note once more that no boundary conditions are prescribed. 

Let $\phi\in C_c^1(\Omega\times[0,T])$ be nonnegative and $u, U$ be the velocities of a local weak and strong solution in a relative neighbourhood of $\supp\phi$, respectively. Then we define for almost every $\tau\in(0,T)$ the \term{localised relative energy} 
\begin{equation*}
E_{rel}(\tau):=\frac{1}{2}\int_\Omega\phi(\tau)|u(x,\tau)-U(x,\tau)|^2dx.
\end{equation*}

We can now state and prove our first result:
\begin{theorem}\label{basicthm}
Let $\phi\in C_c^1(\Omega\times [0,T])$ be a nonnegative test function whose support is contained in a relatively open subset $\mathcal{O}\subset\Omega\times[0,T]$. Let $(u,p)\in L^3_{loc}(\mathcal{O})\times L^{3/2}_{loc}(\mathcal{O})$ be an admissible weak solution in $\Ocal$, and $(U,P)\in C^1(\overline{\mathcal{O}})$ a local strong solution of \eqref{euler}, and assume that $u$ and $U$ coincide on $\Ocal\cap\{t=0\}$, where they take the value $u^0\in C^1(\overline{\Ocal}\cap\{t=0\})$. 
Then, 
\begin{equation*}
\begin{aligned}
E_{rel}^\phi(\tau)&\leq\int_0^\tau\int_\Omega\frac12\partial_t\phi |U-u|^2dxdt+\nabla\phi\cdot\left[\frac{|U-u|^2}{2}u+(P-p)(U-u)\right] dxdt\\
&\hspace{1cm}+2\int_0^\tau\norm{\nabla_{sym}U(t)}_{L^\infty({\Ocal})}E^\phi_{rel}(t)dt.\\
\end{aligned}
\end{equation*}
\end{theorem}
\begin{proof}
We compute 
\begin{align*}
E_{rel}^\phi(\tau)&= \frac{1}{2}\int_{\Omega}\phi(x,\tau)|u(x,\tau)-U(x,\tau)|^2dx\\
&=\frac{1}{2}\int_{\Omega}\phi(x,\tau)|U(x,\tau)|^2dx+\frac{1}{2}\int_{\Omega}\phi(x,\tau)|u(x,\tau)|^2dx-\int_{\Omega}u(x,\tau)\cdot \phi(x,\tau)U(x,\tau)dx\\
&\leq \frac{1}{2}\int_{\Omega}\phi(x,0)|u^0|^2dx+\int_0^\tau\int_\Omega\partial_t\phi\frac12 |U|^2+\left(\frac{|U|^2}{2}+P\right)U\cdot\nabla\phi dxdt\\
&\hspace{1cm} + \frac{1}{2}\int_{\Omega}\phi(x,0)|u^0|^2dx+\int_0^\tau\int_\Omega\partial_t\phi\frac12 |u|^2+\left(\frac{|u|^2}{2}+p\right)u\cdot\nabla\phi dxdt\\
&\hspace{1cm}-\int_{\Omega}u(x,\tau)\cdot \phi(x,\tau)U(x,\tau)dx,
\end{align*}
where we used the local energy inequality. Next, let us employ the weak formulation of the Euler equations for $u$ with test function $\phi U$ to deal with the last integral:

\begin{align*}
E_{rel}^\phi(\tau)&\leq\int_{\Omega}\phi(x,0)|u^0(x)|^2dx-\int_{\Omega}\phi(x,0)u^0(x)\cdot u^0(x)dx\\
&\hspace{1cm}\int_0^\tau\int_\Omega\partial_t\phi\frac12 |U|^2+\left(\frac{|U|^2}{2}+P\right)U\cdot\nabla\phi dxdt\\
&\hspace{1cm} +\int_0^\tau\int_\Omega\partial_t\phi\frac12 |u|^2+\left(\frac{|u|^2}{2}+p\right)u\cdot\nabla\phi dxdt\\
&\hspace{1cm}-\int_0^\tau\int_{\Omega}[\partial_t(\phi U)\cdot u+\nabla (\phi U):(u\otimes u)+\diverg(\phi U)p]dxdt\\
&=\int_0^\tau\int_\Omega\partial_t\phi\left(\frac12 |U|^2+\frac12 |u|^2\right)+\left(\frac{|U|^2}{2}+P\right)U\cdot\nabla\phi +\left(\frac{|u|^2}{2}+p\right)u\cdot\nabla\phi dxdt\\
&\hspace{1cm}-\int_0^\tau\int_{\Omega}\phi[\partial_tU\cdot u+\nabla U:(u\otimes u)]dxdt\\
&\hspace{1cm}-\int_0^\tau\int_{\Omega}\partial_t\phi U\cdot u+\nabla\phi \cdot(u\otimes u)U+\nabla\phi\cdot pUdxdt.
\end{align*}
Now we collect some terms and use the pointwise form of the Euler equations for $U$:
\begin{align*}
E_{rel}^\phi(\tau)&\leq \int_0^\tau\int_\Omega\frac12\partial_t\phi |U-u|^2+\nabla\phi\cdot\left[\left(\frac{|U|^2}{2}+P-p\right)U - (u\otimes u)U+\left(\frac{|u|^2}{2}+p\right)u\right] dxdt\\
&\hspace{1cm}+\int_0^\tau\int_{\Omega}\phi[\diverg(U\otimes U)\cdot u+\nabla P\cdot u-\nabla U:(u\otimes u)]dxdt.
\end{align*}
Observe the identities
\begin{equation*}
\diverg(U\otimes U)=\nabla UU,
\end{equation*} 
\begin{equation*}
\nabla U:(u\otimes u)=u\cdot\nabla U u,
\end{equation*}
and
\begin{equation*}
\int_\Omega\phi\nabla P\cdot u dx=-\int_\Omega\nabla\phi\cdot Pu dx,
\end{equation*}
which follow from the divergence-free property of $u$ and $U$, to obtain
\begin{align*}
E_{rel}^\phi(\tau)&\leq \int_0^\tau\int_\Omega\frac12\partial_t\phi |U-u|^2+\nabla\phi\cdot\left[\frac{|U|^2}{2}U+(P-p)(U-u) - (u\otimes u)U+\frac{|u|^2}{2}u\right] dxdt\\
&\hspace{1cm}+\int_0^\tau\int_{\Omega}\phi[(u-U)\cdot\nabla U(U-u)]dxdt\\
&\hspace{1cm}+\int_0^\tau\int_{\Omega}\phi[U\cdot\nabla U(U-u)]dxdt.
\end{align*}
Using again the divergence-free condition, the last integral is easily seen to equal
\begin{equation*}
-\frac12 \int_0^\tau\int_\Omega\nabla\phi\cdot (U-u)|U|^2dxdt,
\end{equation*}
so that we finally obtain
\begin{align*}
E_{rel}^\phi(\tau)&\leq \int_0^\tau\int_\Omega\frac12\partial_t\phi |U-u|^2dxdt+\nabla\phi\cdot\left[\frac{|U|^2}{2}u+(P-p)(U-u) - (u\otimes u)U+\frac{|u|^2}{2}u\right] dxdt\\
&\hspace{1cm}+2\int_0^\tau\norm{\nabla_{sym}U(t)}_{L^\infty({\Ocal})}E^\phi_{rel}(t)dt\\
&=\int_0^\tau\int_\Omega\frac12\partial_t\phi |U-u|^2dxdt+\nabla\phi\cdot\left[\frac{|U-u|^2}{2}u+(P-p)(U-u)\right] dxdt\\
&\hspace{1cm}+2\int_0^\tau\norm{\nabla_{sym}U(t)}_{L^\infty({\Ocal})}E^\phi_{rel}(t)dt.\\
\end{align*}

\end{proof}

\subsection{The compressible case}

We now consider the isentropic system~\eqref{isentropic}. Again, a pointwise solution $(R,U)\in C^1(\overline{\Ocal})$ is called a \term{local strong solution} in $\Ocal$, and a pair $(\rho,u)$ satisfying the definition of a weak solution and the local energy inequality tested against functions with support in $\Ocal$ is called an \term{admissible weak solution in $\Ocal$}. 

We define for almost every\ $\tau\in[0,T]$ and any $\phi\in C_c^1(\Omega\times[0,T])$ with $\phi\geq0$ the \term{localised relative energy} between $(R,U)$ and $(\rho,u)$ as
\begin{equation*}
E^\phi_{rel}(\tau)=\int_{\Omega}\frac{1}{2}\phi{\rho|u-U|^2}+{\frac{1}{\gamma-1}\phi\rho^\gamma-\frac{\gamma}{\gamma-1}\phi R^{\gamma-1}\rho+\phi R^\gamma}dx.
\end{equation*}

\begin{theorem}\label{Eweak-strong}
Let $\phi\in C_c^1(\Omega\times [0,T])$ be a nonnegative function whose support is contained in a relatively open set $\mathcal{O}\subset\Omega\times[0,T]$. Let $(\rho, u)$ be an admissible weak solution in $\Ocal$, and $(R,U)$ a local strong solution of \eqref{isentropic} in $\Ocal$, and assume that the initial data for $(\rho,u)$ and $(R,U)$ coincide on $\Ocal\cap\{t=0\}$, where they are denoted $(\rho^0,u^0)$. 
Then, for almost every $\tau\in(0,T)$,
\begin{equation}\label{comprelenergy}
\begin{aligned}
E_{rel}^\phi(\tau)&\leq\int_0^\tau\int_\Omega\partial_t\phi\left[\frac12\rho|u-U|^2+R^\gamma-\frac{\gamma}{\gamma-1}\rho R^{\gamma-1}+\frac{1}{\gamma-1}\rho^\gamma\right] dxdt\\
&+\int_0^\tau\int_\Omega\nabla\phi\cdot\left[\frac12\rho|u-U|^2u-\frac{\gamma}{\gamma-1}(R^{\gamma-1}-\rho^{\gamma-1})\rho u+(R^\gamma-\rho^\gamma)U\right] dxdt\\
&\hspace{1cm}+2\int_0^\tau\norm{U(t)}_{C^1(\overline{\Ocal})}E^\phi_{rel}(t)dt.\\
\end{aligned}
\end{equation}
\end{theorem}

\begin{proof}
We set $\eta=\phi U$ in the momentum equation (the second equation of~\eqref{Emass_momentum}) in order to obtain 
\begin{equation}\label{Etest1}
\begin{aligned}
\int_{\Omega}{\rho u}\cdot \phi U(\tau)dx=&\int_{\Omega}\rho^0\phi(0)|u^0|^2dx+\int_0^\tau\int_{\Omega}{\rho u}\cdot\partial_t(\phi U)+{(\rho u\otimes u)}:\nabla (\phi U) dxdt\\
&+\int_0^\tau\int_{\Omega}{\rho^\gamma}\diverg (\phi U) dxdt\\
&=\int_{\Omega}\rho^0\phi(0)|u^0|^2dx+\int_0^\tau\int_{\Omega}{\rho u}\cdot\phi \partial_tU+{(\rho u\otimes u)}:\phi \nabla U dxdt\\
&+\int_0^\tau\int_{\Omega}{\rho^\gamma} \phi \diverg U dxdt\\
&+\int_0^\tau\int_{\Omega}{\rho u}\cdot\partial_t\phi U+{(\rho u\otimes u)}:(\nabla \phi\otimes U) dxdt+\int_0^\tau\int_{\Omega}{\rho^\gamma} \nabla \phi \cdot U dxdt.\\
\end{aligned}
\end{equation}
Likewise, setting $\psi=\frac{1}{2}\phi|U|^2$ and then $\psi=\gamma\phi R^{\gamma-1}$ in the first equation of~\eqref{Emass_momentum} yields
\begin{equation}\label{Etest2}
\begin{aligned}
\frac{1}{2}\int_{\Omega}\phi(\tau)|U(\tau)|^2{\rho}(\tau,x)dx=\int_0^\tau\int_{\Omega}&\phi U\cdot\partial_tU {\rho}+\phi\rho\nabla  U u\cdot Udxdt\\
&+\frac{1}{2}\int_0^\tau\int_{\Omega}\partial_t\phi |U|^2 {\rho}+|U|^2\nabla\phi\cdot{\rho u}dxdt\\
&+\int_{\Omega}\frac{1}{2}\phi(0)|u^0|^2\rho^0dx
\end{aligned}
\end{equation}
and
\begin{equation}\label{Etest3}
\begin{aligned}
\int_{\Omega}\gamma \phi(\tau) R^{\gamma-1}(\tau){\rho}(\tau)dx=&\int_0^\tau\int_{\Omega}\gamma(\gamma-1)\phi R^{\gamma-2}\partial_tR {\rho}+\gamma(\gamma-1)\phi R^{\gamma-2}\nabla R \cdot{\rho u}dxdt\\
&+\int_0^\tau\int_{\Omega}\gamma R^{\gamma-1}\partial_t\phi {\rho}+\gamma R^{\gamma-1}\nabla \phi \cdot{\rho u}dxdt\\
&+\int_{\Omega}\gamma\phi(0) (\rho^0)^\gamma dx,
\end{aligned}
\end{equation}
respectively. 

Then, write the relative energy as
\begin{equation*}
\begin{aligned}
E^\phi_{rel}(\tau)&=\int_{\Omega}\frac{1}{2}\phi{\rho|u|^2}+\frac{1}{\gamma-1}\phi{\rho^\gamma}dx + \int_{\Omega}\phi R^\gamma dx+\frac{1}{2}\int_{\Omega}\phi|U|^2{\rho}dx-\int_{\Omega}\phi U\cdot {\rho u}dx-\int_{\Omega}\frac{\gamma}{\gamma-1}\phi R^{\gamma-1}{\rho}dx\\
&=E^\phi(\tau)+ \int_{\Omega}\phi R^\gamma dx+\frac{1}{2}\int_{\Omega}\phi|U|^2{\rho}dx-\int_{\Omega}\phi U\cdot{\rho u}dx-\int_{\Omega}\frac{\gamma}{\gamma-1}\phi R^{\gamma-1}\rho dx
\end{aligned}
\end{equation*}
(all integrands taken at time $\tau$). Now, using the balances~\eqref{Etest1},~\eqref{Etest2},~\eqref{Etest3} for the last three integrals, we obtain
\begin{equation*}
\begin{aligned}
E^\phi_{rel}(\tau)=E^\phi(\tau)&+\int_{\Omega}\phi R^\gamma dx\\
&+\int_0^\tau\int_{\Omega}\phi U\cdot\partial_tU \rho+\phi\rho\nabla  U u\cdot Udxdt+\int_{\Omega}\frac{1}{2}\phi(0)|u^0|^2\rho^0dx\\ 
&+\frac12\int_0^\tau\int_{\Omega}\partial_t\phi |U|^2 \rho+|U|^2\nabla \phi\cdot{\rho u}dxdt\\ 
&-\int_{\Omega}\phi(0)\rho^0|u^0|^2dx-\int_0^\tau\int_{\Omega}{\rho u}\cdot\phi\partial_t U+{(\rho u\otimes u)}:\phi\nabla U dxdt\\
&-\int_0^\tau\int_{\Omega}\phi{\rho^\gamma}\diverg U dxdt\\
&-\int_0^\tau\int_{\Omega}{\rho u}\cdot\partial_t\phi U+{(\rho u\otimes u)}:(\nabla \phi\otimes U) dxdt-\int_0^\tau\int_{\Omega}{\rho^\gamma} \nabla \phi \cdot U dxdt\\
&-\int_0^\tau\int_{\Omega}\gamma\phi R^{\gamma-2}\partial_tR {\rho}+\gamma\phi R^{\gamma-2}\nabla R \cdot{\rho u}dxdt\\
&-\int_0^\tau\int_{\Omega}\frac{\gamma}{\gamma-1} R^{\gamma-1}\partial_t\phi {\rho}+\frac{\gamma}{\gamma-1} R^{\gamma-1}\nabla \phi \cdot{\rho u}dxdt\\
&-\int_{\Omega}\frac{\gamma}{\gamma-1}\phi(0) (\rho^0)^\gamma dx,
\end{aligned}
\end{equation*}
and using~\eqref{Emvsenergy} we see, for a.e.\ $\tau$,
\begin{equation}\label{Eintermediatestep}
\begin{aligned}
E^\phi_{rel}(\tau)\leq &\int_{\Omega}\phi R^\gamma dx\\
&+\int_0^\tau\int_{\Omega}\phi U\cdot\partial_tU \rho+\phi\rho\nabla  U u\cdot Udxdt\\ 
&+\frac12\int_0^\tau\int_{\Omega}\partial_t\phi |U|^2 \rho+|U|^2\nabla \phi\cdot{\rho u}dxdt\\ 
&-\int_0^\tau\int_{\Omega}{\rho u}\cdot\phi\partial_t U+{(\rho u\otimes u)}:\phi\nabla U dxdt\\
&-\int_0^\tau\int_{\Omega}\phi{\rho^\gamma}\diverg U dxdt\\
&-\int_0^\tau\int_{\Omega}{\rho u}\cdot\partial_t\phi U+{(\rho u\otimes u)}:(\nabla \phi\otimes U) dxdt-\int_0^\tau\int_{\Omega}{\rho^\gamma} \nabla \phi \cdot U dxdt\\
&-\int_0^\tau\int_{\Omega}\gamma\phi R^{\gamma-2}\partial_tR {\rho}+\gamma\phi R^{\gamma-2}\nabla R \cdot{\rho u}dxdt\\
&-\int_0^\tau\int_{\Omega}\frac{\gamma}{\gamma-1} R^{\gamma-1}\partial_t\phi {\rho}+\frac{\gamma}{\gamma-1} R^{\gamma-1}\nabla \phi \cdot{\rho u}dxdt\\
&-\int_{\Omega}\phi(0) (\rho^0)^\gamma dx\\
&+\int_0^\tau\int_\Omega\partial_t\phi\left(\frac{\rho|u|^2}{2}+\frac{1}{\gamma-1}\rho^\gamma\right)+\left(\frac{\rho|u|^2}{2}+\frac{\gamma}{\gamma-1}\rho^\gamma\right)\nabla\phi\cdot u dxdt
\end{aligned}
\end{equation}
Let us collect some terms and write
\begin{equation}\label{Eequality1}
\begin{aligned}
\int_{\Omega}&\phi R^\gamma dx -\int_{\Omega}\phi(0)(\rho^0)^\gamma dx-\int_0^\tau\int_{\T^d}\gamma\phi R^{\gamma-2}\partial_tR \rho dxdt\\
&=\int_0^\tau\int_{\Omega}\phi\partial_t(R^\gamma)-\gamma \phi R^{\gamma-2}\partial_tR \rho +\partial_t\phi R^\gamma dxdt\\
&=\int_0^\tau\int_{\Omega}\gamma\phi R^{\gamma-1}\partial_tR-\gamma\phi R^{\gamma-2}\partial_tR \rho +\partial_t\phi R^\gamma dxdt\\
&=\int_0^\tau\int_{\Omega}\gamma\phi R^{\gamma-2}\partial_tR(R-\rho)+\partial_t\phi R^\gamma dxdt
\end{aligned}
\end{equation}
and
\begin{equation}\label{Eequality3}
\begin{aligned}
\int_0^\tau\int_{\Omega}&\phi U\cdot\partial_tU \rho+\phi\rho\nabla  U u\cdot U-\phi{\rho u}\cdot\partial_t U-{\phi(\rho u\otimes u)}:\nabla U dxdt\\
&=\int_0^\tau\int_{\Omega}\phi\partial_tU\cdot {\rho (U-u)}+\phi\nabla U :{(\rho (U-u)\otimes u)}dxdt.
\end{aligned}
\end{equation}

Insert equalities~\eqref{Eequality1} and~\eqref{Eequality3} into~\eqref{Eintermediatestep} to arrive at
\begin{equation}\label{Eintermediatestep2}
\begin{aligned}
E^\phi_{rel}(\tau)\leq &\int_0^\tau\int_{\Omega}\gamma\phi R^{\gamma-2}\partial_tR(R-\rho)+\partial_t\phi R^\gamma dxdt\\
&+\int_0^\tau\int_{\Omega}\phi\partial_tU\cdot {\rho (U-u)}+\phi\nabla U :{(\rho (U-u)\otimes u)}dxdt\\
&+\frac12\int_0^\tau\int_{\Omega}\partial_t\phi |U|^2 \rho+|U|^2\nabla \phi\cdot{\rho u}dxdt\\ 
&-\int_0^\tau\int_{\Omega}\phi{\rho^\gamma}\diverg U dxdt\\
&-\int_0^\tau\int_{\Omega}{\rho u}\cdot\partial_t\phi U+{(\rho u\otimes u)}:(\nabla \phi\otimes U) dxdt-\int_0^\tau\int_{\Omega}{\rho^\gamma} \nabla \phi \cdot U dxdt\\
&-\int_0^\tau\int_{\Omega}\gamma\phi R^{\gamma-2}\nabla R \cdot{\rho u}dxdt\\
&-\int_0^\tau\int_{\Omega}\frac{\gamma}{\gamma-1} R^{\gamma-1}\partial_t\phi {\rho}+\frac{\gamma}{\gamma-1} R^{\gamma-1}\nabla \phi \cdot{\rho u}dxdt\\
&+\int_0^\tau\int_\Omega\partial_t\phi\left(\frac{\rho|u|^2}{2}+\frac{1}{\gamma-1}\rho^\gamma\right)+\left(\frac{\rho|u|^2}{2}+\frac{\gamma}{\gamma-1}\rho^\gamma\right)\nabla\phi\cdot u dxdt
\end{aligned}
\end{equation}

Furthermore, using the divergence theorem, we have
\begin{equation*}
\begin{aligned}
-\int_0^\tau\int_{\Omega}&\phi{\rho^\gamma}\diverg U dxdt-\int_0^\tau\int_{\Omega}\gamma\phi R^{\gamma-2}\nabla R \cdot{\rho u}dxdt\\
&=\int_0^\tau\int_{\Omega}-\phi{\rho^\gamma}\diverg U +\gamma \phi R^{\gamma-2}\nabla R \cdot(RU-{\rho u})-\gamma\phi R^{\gamma-2}\nabla R\cdot RUdxdt\\
&=\int_0^\tau\int_{\Omega}\phi(R^\gamma-{\rho^\gamma})\diverg U +\gamma\phi R^{\gamma-2}\nabla R \cdot(RU-{\rho u})+R^\gamma\nabla\phi\cdot Udxdt.\\
\end{aligned}
\end{equation*}
Inserting this back into~\eqref{Eintermediatestep2} and observing that, by the mass equation for $(R,U)$, 
\begin{equation*}
\begin{aligned}
\gamma &\phi R^{\gamma-2}\partial_tR(R-\rho)+\gamma\phi R^{\gamma-2}\diverg UR(R-\rho) +\gamma\phi R^{\gamma-2}\nabla R \cdot RU
=\gamma\phi R^{\gamma-2}U\cdot\nabla R\rho,
\end{aligned}
\end{equation*}
we obtain
\begin{equation}\label{Eintermediatestep3}
\begin{aligned}
E^\phi_{rel}(\tau)\leq &\int_0^\tau\int_{\Omega}\partial_t\phi R^\gamma dxdt\\
&+\int_0^\tau\int_{\Omega}\phi\partial_tU\cdot {\rho (U-u)}+\phi\nabla U :{(\rho (U-u)\otimes u)}dxdt\\
&+\frac12\int_0^\tau\int_{\Omega}\partial_t\phi |U|^2 \rho+|U|^2\nabla \phi\cdot{\rho u}dxdt\\ 
&+\int_0^\tau\int_{\Omega}\phi(R^\gamma-{\rho^\gamma})\diverg U dxdt+\int_0^\tau\int_{\Omega}R^\gamma\nabla\phi\cdot Udxdt\\
&-\int_0^\tau\int_{\Omega}{\rho u}\cdot\partial_t\phi U+{(\rho u\otimes u)}:(\nabla \phi\otimes U) dxdt-\int_0^\tau\int_{\Omega}{\rho^\gamma} \nabla \phi \cdot U dxdt\\
&+\int_0^\tau\int_\Omega \gamma \phi R^{\gamma-2}\rho\nabla R\cdot(U-u)-\gamma\phi R^{\gamma-1}\diverg U (R-\rho)dxdt\\
&-\int_0^\tau\int_{\Omega}\frac{\gamma}{\gamma-1} R^{\gamma-1}\partial_t\phi {\rho}+\frac{\gamma}{\gamma-1} R^{\gamma-1}\nabla \phi \cdot{\rho u}dxdt\\
&+\int_0^\tau\int_\Omega\partial_t\phi\left(\frac{\rho|u|^2}{2}+\frac{1}{\gamma-1}\rho^\gamma\right)+\left(\frac{\rho|u|^2}{2}+\frac{\gamma}{\gamma-1}\rho^\gamma\right)\nabla\phi\cdot u dxdt.
\end{aligned}
\end{equation}

The expression in the second line is rewritten as 
\begin{equation}\label{Epointwiseest}
\begin{aligned}
\phi\partial_tU&\cdot {\rho(U-u)}+\phi\nabla U :{(\rho (U-u)\otimes u)}\\
&=\phi\partial_tU\cdot {\rho(U-u)}+\phi\nabla U :{(\rho (U-u)\otimes U)}+\phi\nabla U :{(\rho(U-u)\otimes(u-U))},
\end{aligned}
\end{equation}
and the integral of the last expression as well as the sum of the two integrals involving $\diverg U$ in~\eqref{Eintermediatestep3} can be estimated by
\begin{equation}\label{Eabsorb}
C\norm{U}_{C^1}\int_0^\tau E^\phi_{rel}(t)dt.
\end{equation}
For the other terms in~\eqref{Epointwiseest} we get, invoking the momentum equation for $(R,U)$,
\begin{equation}\label{Emomentum}
\begin{aligned}
\phi\partial_tU&\cdot {\rho(U-u)}+\phi\nabla U :(\rho(U-u)\otimes{U})\\
&=\phi\frac{1}{R}(\partial_t(RU)+\diverg(RU\otimes U))\cdot{\rho(U-u)}\\
&= - \gamma\phi R^{\gamma-2}\nabla R\cdot{\rho(U-u)}.
\end{aligned}
\end{equation}
Putting together~\eqref{Eintermediatestep3},~\eqref{Eabsorb}, and~\eqref{Emomentum}, we obtain
\begin{equation*}
E_{rel}(\tau)\leq C\norm{U}_{C^1}\int_0^\tau E_{rel}(t)dt+\int_0^\tau \int_\Omega a(x,t)\partial_t\phi(x,t)+b(x,t)\cdot\nabla\phi(x,t)dxdt,
\end{equation*}
where
\begin{equation*}
\begin{aligned}
a(x,t)&=R^\gamma+\frac12\rho|U|^2-\rho u\cdot U-\frac{\gamma}{\gamma-1}\rho R^{\gamma-1}+\frac12 \rho|u|^2+\frac{1}{\gamma-1}\rho^\gamma\\
&=\frac12\rho|u-U|^2+R^\gamma-\frac{\gamma}{\gamma-1}\rho R^{\gamma-1}+\frac{1}{\gamma-1}\rho^\gamma
\end{aligned}
\end{equation*}
and
\begin{equation*}
\begin{aligned}
b(x,t)&=\frac12|U|^2\rho u+R^\gamma U-\rho u\otimes uU-\rho^\gamma U-\frac{\gamma}{\gamma-1}R^{\gamma-1}\rho u+\frac12\rho|u|^2u+\frac{\gamma}{\gamma-1}\rho^\gamma u\\
&=\frac12\rho|u-U|^2u-\frac{\gamma}{\gamma-1}(R^{\gamma-1}-\rho^{\gamma-1})\rho u+(R^\gamma-\rho^\gamma)U.
\end{aligned}
\end{equation*}
\end{proof}


\section{Applications}\label{appl}
\subsection{The incompressible case}
\begin{corollary}[of Theorem \ref{basicthm}]\label{incompcor}
Under the hypotheses of Theorem~\ref{basicthm},
\begin{equation*}
\liminf_{\tau\searrow 0}E^\phi_{rel}(\tau)=0.
\end{equation*}
\end{corollary}

\begin{proof}
Suppose, in contrast, $\liminf_{\tau\searrow 0}E^\phi_{rel}(\tau)>0$, so that there exists a time $T_0>0$ and $\delta>0$ such that $E^\phi_{rel}(\tau)\geq\delta$ for almost all $\tau<T_0$. Since, by assumption, 
\begin{equation*}
g(x,t):=\frac12\partial_t\phi |U-u|^2dxdt+\nabla\phi\cdot\left[\frac{|U-u|^2}{2}u+(P-p)(U-u)\right] \in L^1(\R^d\times(0,T_0)),
\end{equation*}
by Theorem \ref{basicthm} there is a constant $C(\delta)$ such that
\begin{equation*}
\begin{aligned}
E^\phi_{rel}(\tau)&\leq C(\delta) \int_0^\tau \left[\delta\int_{\R^d}g(x,t)dx+\norm{\nabla_{sym}U(t)}_{L^\infty}E^\phi_{rel}(t)\right] dt\\
&\leq\int_0^\tau G(t) E^\phi_{rel}(t)dt\quad\text{for almost every $\tau\in(0,T_0)$,}
\end{aligned}
\end{equation*} 
where $G\in L^1(0,T_0)$. Gronwall's inequality now implies $E^\phi_{rel}\equiv0$, which yields the desired contradiction.
\end{proof}

\subsection{The compressible case}
First, we obtain a similar result as the previous corollary for the isentropic system:
\begin{corollary}[of Theorem \ref{Eweak-strong}]\label{compcor}
Under the hypotheses of Theorem~\ref{Eweak-strong},
\begin{equation*}
\liminf_{\tau\searrow 0}E^\phi_{rel}(\tau)=0.
\end{equation*}
\end{corollary}
The proof is analogous to the proof of Corollary~\ref{incompcor}.

For \emph{bounded} solutions, we can say much more. Let us denote
\begin{equation*}
\begin{aligned}
A(R,U;\rho,u)=\frac12\rho|u-U|^2+R^\gamma-\frac{\gamma}{\gamma-1}\rho R^{\gamma-1}+\frac{1}{\gamma-1}\rho^\gamma
\end{aligned}
\end{equation*}
and
\begin{equation*}
\begin{aligned}
B(R,U;\rho,u)=\frac12\rho|u-U|^2u-\frac{\gamma}{\gamma-1}(R^{\gamma-1}-\rho^{\gamma-1})\rho u+(R^\gamma-\rho^\gamma)U.
\end{aligned}
\end{equation*}

\begin{lemma}\label{ABestimate}
Let $0\leq\underline{r}<\overline{r}<\infty$ and $v>0$. If $\gamma<2$, assume in addition $\underline{r}>0$. Then there exists a number $C$ depending only on $\underline{r}$, $\overline{r}$, $v$, and $\gamma>1$ such that
\begin{equation*}
|B(R,U;\rho,u)|\leq C A(R,U;\rho,u)
\end{equation*}
whenever $R,\rho\in[\underline{r},\overline{r}]$ and $|U|, |u|\leq v$.
\end{lemma}

\begin{proof}
Since $|u|\leq v$ and the potential part of $A(R,U;\rho,u)$ is nonnegative, one sees immediately that 
\begin{equation*}
\left|\frac12\rho|u-U|^2u\right| \leq C A(R,U;\rho,u).
\end{equation*}
For the other terms, we write
\begin{equation*}
\begin{aligned}
&\left|-\frac{\gamma}{\gamma-1}(R^{\gamma-1}-\rho^{\gamma-1})\rho u+(R^\gamma-\rho^\gamma)U\right|\\
&\leq \left|-\frac{\gamma}{\gamma-1}(R^{\gamma-1}-\rho^{\gamma-1})\rho U+(R^\gamma-\rho^\gamma)U\right|+\frac{\gamma}{\gamma-1}\left|(R^{\gamma-1}-\rho^{\gamma-1})\rho (U-u)\right|\\
&=  \left(R^\gamma-\frac{\gamma}{\gamma-1}\rho R^{\gamma-1}+\frac{1}{\gamma-1}\rho^\gamma\right)|U|+\frac{\gamma}{\gamma-1}\left|(R^{\gamma-1}-\rho^{\gamma-1})\rho (U-u)\right|.
\end{aligned}
\end{equation*}
The first term ist precisely the potential part of $A(R,U;\rho,u)$ times $|U|$, which is less than or equal to $v$. Hence it remains to control the second term. Observe first that
\begin{equation*}
\left|(R^{\gamma-1}-\rho^{\gamma-1})\rho (U-u)\right|\leq \frac12 \rho(R^{\gamma-1}-\rho^{\gamma-1})^2+\frac12\rho|U-u|^2,
\end{equation*}
so that the latter summand is exactly the kinetic part of $A(R,U;\rho,u)$. Moreover, since $\rho\leq \overline{r}$, we only need to handle the term 
\begin{equation}\label{criticalterm}
(R^{\gamma-1}-\rho^{\gamma-1})^2.
\end{equation}
To this end, note that for $\underline{r}\leq\rho,R\leq\overline{r}$ it holds that
\begin{equation*}
R^\gamma-\frac{\gamma}{\gamma-1}\rho R^{\gamma-1}+\frac{1}{\gamma-1}\rho^\gamma\geq c(R-\rho)^2.
\end{equation*}
Indeed, this follows from the smoothness and strict convexity of the map $r\mapsto r^\gamma$ on $[\underline{r},\overline{r}]$ and the observation that the expression on the left hand side contains exactly the terms of order $\geq 2$ of the Taylor expansion of this map around $R$; cf.\ also~\cite{FeJiNo}. 

Then, writing $\lambda=\frac{\rho}{R}$, we estimate
\begin{equation*}
\begin{aligned}
(R^{\gamma-1}-\rho^{\gamma-1})^2&=R^{2\gamma-2}(1-\lambda^{\gamma-1})^2\\
&= R^{2\gamma-4} \left(\frac{1-\lambda^{\gamma-1}}{1-\lambda}\right)^2(R-\rho)^2\leq C(R-\rho)^2,
\end{aligned}
\end{equation*}
thus concluding the proof. Indeed, for the last inequality, we use that the map
\begin{equation*}
\lambda\mapsto\frac{1-\lambda^{\gamma-1}}{1-\lambda}
\end{equation*}
is continuous on any compact subinterval of $(0,\infty)$; for $\lambda\neq1$ this is obvious, and at $\lambda=1$ it follows from 
\begin{equation*}
\lim_{\lambda\to 1}\frac{1-\lambda^{\gamma-1}}{1-\lambda}=\gamma-1
\end{equation*}
thanks to L'H\^opital's rule.
\end{proof}

With this preparation, the local weak-strong uniqueness for the compressible Euler system now follows easily from Theorem~\ref{Eweak-strong} and Lemma~\ref{ABestimate}:

\begin{theorem}[Local weak-strong uniqueness]\label{locweakstrong}
Let ${\Omega}\times [0,T]$ be a domain in which $(R,U)$ is a local strong solution of \eqref{isentropic}, and $(\rho,u)$ is an admissible weak solution of \eqref{isentropic} with initial data $(R(0),U(0))$ in $\Omega$, satisfying the local energy inequality on $\Omega\times[0,T]$. Suppose further there exist numbers $0\leq \underline{r}<\overline{r}<\infty$ (if $\gamma<2$ assume $\underline{r}>0$) and $v>0$ such that
\begin{equation*}
\underline{r}\leq R,\rho\leq\overline{r}\quad\text{and}\quad |U|, |u|\leq v\quad\text{almost everywhere in $\Omega\times(0,T)$.}
\end{equation*}

Then for every point $x_0\in{\Omega}$ there exist another domain $x_0\in\Omega'\subset\Omega$ and a time $0<T'\leq T$ such that
\begin{equation*}
R=\rho=0\quad\text{or}\quad R=\rho, \quad U=u\quad\text{almost everywhere in $\Omega'\times(0,T')$.}
\end{equation*}
\end{theorem}

\begin{proof}
We will choose $T'>0$ and $\phi\in C_c^1(\Omega\times[0,T'])$ with the following properties:
\begin{itemize}
\item[i)] $\phi\geq0$;
\item[ii)] There exists an open set $\Omega'\ni x_0$ such that $\phi\equiv1$ in $\Omega'\times[0,T']$;
\item[iii)] The sum of the first two integrals on the right hand side of \eqref{comprelenergy} becomes non-positive for every $\tau\leq T'$.
\end{itemize}
Once this is achieved, Gronwall's inequality applied to \eqref{comprelenergy} will imply $E_{rel}^\phi\equiv0$ in $(0,T')$ and hence the statement of the theorem. Indeed, note that, as $\gamma>1$, the map $|\cdot|^\gamma$ is strictly convex, which implies that the relative energy is always non-negative, being zero if and only if $\rho\equiv R$ and, if $\rho=R>0$, $u\equiv U$ on the support of $\phi$. Thus $E^\phi_{rel}(\tau)=0$ for a.e.\ $\tau$ implies Theorem~\ref{locweakstrong}. 

To define $\phi$ appropriately, let $\eta>0$ such that $B_\eta(x_0)\subset\Omega$, and let $\phi^0\in C_c^1(B_\eta(x_0))$ be a nonnegative function, radially symmetric around $x_0$, monotone non-increasing in the radial direction, with $\phi^0\equiv 1$ in $B_{\eta/2}(x_0)$.

Set 
\begin{equation*}
a(x,t)=A(R(x,t),U(x,t);\rho(x,t),u(x,t)),\quad b(x,t)=B(R(x,t),U(x,t);\rho(x,t),u(x,t))
\end{equation*}
with the notation from the beginning of this subsection. Then, by Lemma~\ref{ABestimate} and the assumptions of the theorem, there exists $C>0$ such that 
\begin{equation*}
|b(x,t)|\leq Ca(x,t)\quad \text{for almost every $(x,t)\in\Omega\times(0,T)$.}
\end{equation*}
Now set $\phi$ to be
\begin{equation*}
\phi(x,t)=\phi^0\left(x+Ct\frac{x-x_0}{|x-x_0|}\right),
\end{equation*}
so that $\phi$ solves the transport equation 
\begin{equation}\label{transport}
\partial_t\phi-C\frac{x-x_0}{|x-x_0|}\cdot\nabla\phi=0
\end{equation}
with initial datum $\phi^0$. By choice of $\phi^0$, we have that $\phi(t)\in C_c^1(B_\eta(x_0))$ at least for  
\begin{equation*}
0\leq t <\frac{\eta}{2C},
\end{equation*}
and $\phi\equiv 1$ in the set
\begin{equation*}
B_{\eta/4}(x_0)\times [0,T'),\quad T'=\frac{\eta}{4C}.
\end{equation*}
Moreover, by choice of $C$, radial symmetry and monotonicity of $\phi$, and \eqref{transport},
\begin{equation*}
\begin{aligned}
a(x,t)\partial_t\phi(x,t)&+b(x,t)\cdot\nabla\phi(x,t)\\
&= Ca(x,t)\frac{x-x_0}{|x-x_0|}\cdot\nabla\phi(x,t)+b(x,t)\cdot\nabla\phi(x,t)\\
&\leq -Ca(x,t)|\nabla\phi(x,t)|+|b(x,t)||\nabla\phi(x,t)|\leq0
\end{aligned}
\end{equation*}
for almost every $(x,t)\in B_{\eta}(x_0)\times (0,T')$. Hence indeed $\phi$ satisfies i)--iii) for $\Omega'=B_{\eta/4}(x_0)$ and $T'=\frac{\eta}{4C}$, and we conclude.
\end{proof}

This theorem and its proof immediately imply the following corollary:
\begin{corollary}[Finite speed of propagation]\label{finite}
Let $(\rho,u)$ be an admissible weak solution and $(R,U)$ a $C^1$ solution of \eqref{isentropic} on $\R^d\times(0,T)$ and assume  the initial data of both solutions agree outside of a ball $B_\eta(0)$. Suppose further there exist numbers $0\leq\underline{r}<\overline{r}<\infty$ (if $\gamma<2$ assume $\underline{r}>0$) and $v>0$ such that
\begin{equation*}
\underline{r}\leq \rho,R\leq\overline{r}\quad\text{and}\quad |u|, |U|\leq v\quad\text{almost everywhere in $\R^d\times(0,T)$.}
\end{equation*}
Then $(\rho,u)=(R,U)$ on the complement of the Ball $B_{\eta+Ct}(0)$ for almost every $t\in(0,T)$, where $C$ is the constant from Lemma~\ref{ABestimate}.
\end{corollary}
\begin{remark}
One possible choice is $R\equiv\overline{\rho}>0$, $U\equiv{0}$, in which case the corollary yields that an initially compactly supported solution\footnote{Recall that we use ``compactly supported" to mean ``the velocity is zero and the density is constant outside a compact set".} will remain compactly supported, the support spreading with speed at most $C$. This is the most common formulation of the principle of finite propagation speed. However, we have not been able to relate the speed $C$ to the speed of sound.

\end{remark}

\begin{corollary}
Let $(\rho,u)$ be a bounded admissible weak solution of \eqref{isentropic} on $\Omega\times(0,T)$ such that 
\begin{equation*}
\rho, \rho^0\geq\underline{r}\quad\text{almost everywhere in $\R^d\times(0,T)$,} 
\end{equation*}
where $\underline{r}\geq 0$ with strict inequality if $\gamma<2$. If the initial data $(\rho^0,u^0)$ is smooth on some subdomain $\Omega'\subset\Omega$, then for each $x_0\in\Omega'$ there exist $\Omega''\subset\Omega'$ and $0<T'\leq T$ such that $(\rho,u)$ remains smooth in $\Omega''\times[0,T')$.
\end{corollary}
\begin{proof}
Let $(R^0,U^0)$ be a smooth and sufficiently decaying extension of $(\rho^0,u^0)\restriction_{\Omega'}$ to all of $\R^d$, with $R^0\geq\frac12\underline{r}$. Then by classical local well-posedness results (see e.g.\ \cite{majda}), there exists a time $\tilde{T}>0$ and a smooth solution $(R,U)$ on $\R^d\times[0,\tilde{T})$ emanating from $(R^0,U^0)$. Then, by Theorem~\ref{locweakstrong}, $(\rho,u)$ coincides with $(R,U)$ on some domain $\Omega''\times[0,T')$ with $x_0\in\Omega''$.
\end{proof}

\end{document}